\documentclass[12pt]{article}
\usepackage{mathrsfs,amssymb,amsmath,extarrows,bm}
\usepackage{paralist}
\usepackage{graphicx}
\usepackage{color,tabu,diagbox,multirow,caption}
\usepackage{algorithm,algorithmic} 
\usepackage{mathtools} 
\mathtoolsset{showonlyrefs}

\newtheorem{theorem}{Theorem}[section]

\newtheorem{proposition}{Proposition}

\newtheorem{definition}[theorem]{Definition}

\newtheorem{scheme}{Scheme}[section]
\newtheorem{proof}{Proof.}[section]
\newtheorem{remark}{Remark}[section]

\newcommand{\mg}{\mathcal{G}}

\newcommand{\mw}{\mathcal{W}}
\newcommand{\mn}{\mathcal{N}}
\newcommand{\mf}{\mathcal{F}}

\newcommand{\mr}{\mathcal{R}}

\newcommand{\mbn}{\bm{n}}
\newcommand{\mbr}{\bm{r}}
\newcommand{\mbh}{\bm{h}}

\newcommand{\bbR}{\mathbb{R}}

\newcommand{\hphi}{\hat{\phi}}
\newcommand{\bphi}{\bar{\phi}}

\newcommand{\tep}{\tilde{\varepsilon}}
\newcommand{\tal}{\tilde{\alpha}}
\newcommand{\inner}[1]{\left<#1\right>_{AP}}
\newcommand{\norm}[1]{\left\|#1\right\|^{2}_{AP}}

\newcommand\tbbint{{-\mkern -16mu\int}}

\newcommand\dbbint{{-\mkern -19mu\int}}

\newcommand\bbint{
	{\mathchoice{\dbbint}{\tbbint}{\tbbint}{\tbbint}}
}

\begin{document}

	\title%[High-order energy stable schemes for iPFC model]
	{High-order energy stable schemes of incommensurate phase-field crystal model}

	\author{Kai Jiang and Wei Si}
	\date{}
	\maketitle

	\footnotetext{K. Jiang, Email: kaijiang@xtu.edu.cn; \\
		Wei Si, 201610111098@smail.xtu.edu.cn}

	{\footnotesize
		% please put the address of the first author
		\centerline{School of Mathematics and Computational Science,}
		\centerline{Hunan Key Laboratory for Computation and Simulation in Science and Engineering,}
		\centerline{Xiangtan University, Xiangtan, Hunan, P.R. China, 411105.}
	} % Do not forget to end the {\footnotesize by the sign }

%The abstract of your paper
\begin{abstract}
%This is the abstract of your paper and it should not exceed \textbf{200} words.
This article focuses on the development of high-order energy stable schemes for the multi-length-scale 
incommensurate phase-field crystal model which is able to study the phase behavior of aperiodic structures. 
These high-order schemes based on the scalar auxiliary variable (SAV) and spectral 
deferred correction (SDC) approaches are suitable for the $L^2$ gradient flow 
equation, \textit{i.e.}, the Allen-Cahn dynamic equation. 
Concretely, we propose a second-order Crank-Nicolson (CN) scheme of the SAV system,
prove the energy dissipation law, and give the error estimate in the almost periodic
function sense. 
Moreover, we use the SDC method to improve the computational accuracy of the SAV/CN scheme. 
Numerical results demonstrate the advantages of high-order numerical methods in numerical computations 
and show the influence of length-scales on the formation of ordered structures.

\noindent{\it Primary}:

\noindent{\it Keywords}: Incommensurate phase-field crystal model, Scalar auxiliary variable method, 
Spectral deferred correction approach, Allen-Cahn equation, Energy dissipation law,
Error estimate.
\end{abstract}

%\maketitle

\section{Introduction}
\label{sec:intro}

%% What is incommensurate phase-field crystal model?
%% The important role of multi-length-scale phase-field crystal model in describing incommensurate structures;
Aperiodic crystals, such as quasicrystals, are an important class of materials 
whose Fourier spectra cannot be all expressed by a set of basis vectors over the rational number field.
The irrational coefficients give rise to the denseness of Fourier spectra which 
results in the difficulties in the theoretical study.
Theoretically, a multiple characteristic length-scale model which possesses, at least, an
irrational scale, has been widely applied to study the formation and
thermodynamic stability of the aperiodic structures\,\cite{
Bak1985, Jaric1985, Lifshitz1997, Jiang2017, Savitz2018, Jiang2019}. 
%% Give some examples about the Landau theories since it is widely applied;
The early model could trace back to Bak's work on three-dimensional
icosahedral quasicrystals. Since then, many related models have been proposed to study
aperiodic structures, including for multicomponent systems\,\cite{Jiang2019}. 
Among these models, Lifshitz and Petrich (LP) modified the
Swift-Hohenberg model and explicitly added an incommensurate two-length-scale
potential into a Lyapunov functional to explore quasiperiodic patterns that emerged in
Faraday experiments\,\cite{Lifshitz1997}. 
Recently, Savitz et al. extended the LP model from two-length-scale potential to multiple 
($\geq 3$) length-scale potential and studied more kinds of quasicrystals\,\cite{Savitz2018}. 
The $m$-length-scale model actually is an incommensurate multi-length-scale
phase-field-crystal (iPFC) model who owns a $4m$ ($m\in\mathbb{N}$) order
differential operator and nonlinear term in the energy functional.  
A high precision computation is helpful to study the phase behaviors of
aperiodic structures.
In this article, we will pay attention to the development of high-order numerical
methods for the iPFC model. 

%% The necessity of discrete approaches in time direction;
%% The development of time-discrete schemes, mainly focus on SAV;
%% The high-order time-discrete schemes;
%Concretely, we consider the Allen-Cahn dynamic equation in this article. 
%Due to the high-order spatial derivative and the nonlinearity of the time-dependent equation, it is necessary 
%to develop efficient and accurate numerical algorithms.
Recently, various numerical methods have been proposed to solve phase-field equations 
including the convex splitting methods\,\cite{Wise2009}, 
the linear stabilized schemes\,\cite{Shen2010}, the invariant energy quadratization (IEQ)\,\cite{Yang2016} 
and the scalar auxiliary variable (SAV) approaches\,\cite{Shen2018, Li2019}.
The convex splitting method splits the energy functional into the convex and concave parts.
The method treats the convex part implicitly and the concave one explicitly to keep the unconditional 
energy stability. 
While the application of this method is restricted by the form of the energy
functional, such as double-well bulk energy. 
The linear stabilized scheme adds a penalty term to improve its stability and deals with 
the nonlinear terms explicitly for implementing it easily.
However, such a stabilized approach makes it difficult to design second-order unconditionally energy stable schemes.
Assume that the nonlinear part has a lower bound, via introducing an auxiliary variable, the IEQ method 
transforms the energy functional into a quadratic form to keep the unconditional energy dissipation property. 
Similarly, the SAV approach introduces a scalar auxiliary variable by supposing the bounded bulk energy and 
obtains an unconditionally energy stable system. 
Besides these methods, the spectral deferred correction (SDC)\,\cite{Dutt2000,Xu2019} algorithm is an efficient strategy 
to improve the accuracy of the above schemes. 
%\textcolor{blue}{
%The SDC approach adopts a different numerical discrete form which depends on the
%chosen numerical schemes. }
In the paper, we will apply the SAV approach to solve the time-dependent equations and further use the SDC strategy 
to improve the numerical accuracy. 

%% The introduction of the spatial discrete method;
%% The PM is pivotal;
For aperiodic structures, two kinds of numerical methods, including 
the crystalline approximant method (CAM) and the projection method (PM) are
usually used to discretize the quasiperiodic functions\,\cite{Jiang2014}. 
The CAM uses a big periodic structure to approximate an aperiodic structure and 
corresponds to the Diophantine approximation problem which studies 
how to approximate irrational numbers by rational numbers\,\cite{Daven1946}.
To evaluate aperiodic structures accurately, the CAM needs an extremely big computational 
region with an unacceptable computational burden to reduce the error of Diophantine approximation. 
To avoid the Diophantine approximation problem, the PM accurately describes 
aperiodic structures based on the fact that the aperiodic structure can be regarded
as a periodic crystal in an appropriate higher-dimensional space.
The PM uses one higher-dimensional periodic region to capture the essential characteristics and 
greatly reduces computational complexity.

%% The organization of this article;
The rest of the paper is organized as follows.
In Section\,\ref{sec:model}, we first outline some useful preliminaries of the almost
periodic functions and then present the iPFC model.  
The energy dissipation principle of the $L^2$ gradient flow for the iPFC model in the
almost periodic sense is also given. 
In Section\,\ref{sec:theory}, we propose a second-order energy stable scheme for the
iPFC model and give the corresponding error estimate.
And we use the SDC approach to improve its computational accuracy efficiently. 
Section \ref{sec:results} presents the convergence rates of these numerical schemes and discusses the 
advantages of high-order numerical approaches in simulating dynamic evolution. 
Moreover, we also show the influence of multiple length-scales on the thermodynamic stability of 
aperiodic structures. 
There are some conclusions in Section\,\ref{sec:summary}.

\section{Problem formulation}
\label{sec:model}

\subsection{Preliminary}
\label{subsec:AP}
%{Some properties of almost periodic functions}

%% Definition of the almost periodic inner product;
%Due to incommensurate functions are space-filling without decay, it is necessary to define 

Aperiodic structures are space-filling phases without decay.
A useful mathematical theory to describe aperiodic structures is the almost periodic
function theory which is a generalization of continuous periodic functions. We
define the notation of a $d$-dimensional almost periodic function. 
\begin{definition}
	Let $f(\mbr)$ be a real-valued or complex-valued function defined on $\bbR^d$ and
	let $\epsilon > 0$. We say that $\zeta\in\bbR^d$ is an $\epsilon$-almost
	period of $f$ if 
	\begin{align*}
		| f(\mbr-\zeta) - f(\mbr)| < \epsilon, ~\mbox{~~for any ~} \mbr\in\bbR^d.
	\end{align*}
	A function $f$ is \textit{almost periodic} on $\bbR^d$ if it is continuous and if for
	every $\epsilon$ there exists a number $L=L(\epsilon, f)$ such that any
	cube with the side length of $L$ on $\bbR^d$ contains an $\epsilon$-almost period
	of $f$.
\end{definition}
The almost periodic $L^2$ inner product is defined by 
\begin{equation}
	\inner{f,g} = \lim_{R\to\infty} \frac{1}{|Q(R)|} \int_{Q(R)} f(\mbr)g(\mbr)\,d\mbr,
	\label{eq:definite.AP.inner}
\end{equation}
where $Q(R)=[-R,R]^d$ and $|Q(R)|$ is the measure of $Q(R)$. 
It is known that this inner product is well-defined\,\cite{Corduneanu1989}. 
We denote the corresponding norm by $\norm{\cdot} = \inner{\cdot,\cdot}$. 
For simplicity, we denote the average spacial integral over the whole space as 
\begin{equation}
	\bbint =  \lim_{R\to\infty} \frac{1}{|Q(R)|} \int_{Q(R)}.
\end{equation}
Some useful properties of $d$-dimensional almost periodic functions are presented as follows. 
\begin{proposition}
\label{thm:AP}
	~
	\begin{enumerate}[(1).]
		\item An almost periodic function is uniformly continuous and bounded.
		\item If $f(\mbr)$ and $g(\mbr)$ are almost periodic functions, then $f(\mbr) + g(\mbr)$ 
			and $f(\mbr) \cdot g(\mbr)$ are almost periodic functions.
%		\item If $f(\mbr)$ is almost periodic and $\nabla f(\mbr)$ is uniformly continuous, 
%			then $\nabla f(\mbr)$ is almost periodic. 
	\end{enumerate}
\end{proposition}
These properties can be easily proven from the one-dimensional
results\,\cite{Corduneanu1989}.

%% The Green's identity in the almost periodic meaning;
\begin{theorem}
\label{thm:Green}
	If $f(\mbr)$ and $g(\mbr)$ are almost periodic functions and differentiable,
	then the Green's identity holds in the almost periodic sense, \textit{i.e.},
	\begin{equation}
		\left< f(\mbr), \nabla g(\mbr) \right>_{AP} = - \left< \nabla f(\mbr), g(\mbr) \right>_{AP}.
		\label{eq:Green}
	\end{equation}
\end{theorem}
\begin{proof}
	Since $f(\mbr)$ and $g(\mbr)$ are almost periodic functions, then
	there exists a constant $M$ such that $\sup\limits_{\mbr} \left\{
	\left|f(\mbr)\right|, \left|g(\mbr)\right| \right\} \leq M$. 
	Denote 
	\begin{equation}
		b_{R} = \lim_{R\to \infty} \frac{1}{|Q(R)|} \int_{\partial Q(R)} f(\mbr) \mbn
		\cdot g(\mbr)\,ds,
	\end{equation}
	where $\mbn$ is the outward normal of $\partial Q(R)$. 
	In the $d$-dimensional space, we have  
	\begin{equation}
		|b_{R}| \leq \lim_{R\to \infty} \frac{M^{2} 2d(2R)^{d-1}}{(2R)^d} = 0. 
	\end{equation}
	Therefore, we obtain the desired conclusion by 
	\begin{equation}
		\begin{aligned}
			\left< f(\mbr), \nabla g(\mbr) \right>_{AP}
			& = \bbint f(\mbr) \nabla g(\mbr)\,d\mbr
			\\
			& = - \bbint \nabla f(\mbr) g(\mbr)\,d\mbr
			= - \left< \nabla f(\mbr), g(\mbr) \right>_{AP}.
		\end{aligned}
		\label{eq:Green.deduce}
	\end{equation}
\end{proof}
%
%%% The Poincare inequality;
%\begin{theorem}
%\label{thm:Poincare}
%	If $ f(\mbr) $ is an almost periodic function and differentiable,
%	we get the Poincar\'e identity in the almost periodic sense,
%	\begin{equation}
%		\|f(\mbr)\|_{AP} \leq C \|\nabla f(\mbr)\|_{AP}.
%		\label{eq:Poincare}
%	\end{equation}
%\end{theorem}
%%
%\begin{proof}
%	Let $ g(\mbr) $ be an almost periodic function satisfying $-\Delta g(\mbr) = 1$ 
%	with $\sup_{\mbr} |\nabla g(\mbr)| \leq C$.
%	Using Theorem \eqref{thm:Green}, we have
%	\begin{equation}
%		\begin{aligned}
%			\|f(\mbr)\|^{2}_{AP}
%			= \left< f^{2}(\mbr), 1 \right>_{AP}
%			= -\left< f^{2}(\mbr), \Delta g(\mbr) \right>_{AP}
%			%= - \lim_{R\to \infty} \frac{1}{|B_R|} \int_{\partial B_R} f^{2}(\mbr)\mbn \cdot \nabla g(\mbr) ds
%			= \left< 2f(\mbr) \nabla f(\mbr), \nabla g(\mbr) \right>_{AP}.
%		\end{aligned}
%	\end{equation}
%	Thus leads to 
%	\begin{equation}
%		\|f(\mbr)\|^{2}_{AP} \leq 2 \max |\nabla g(\mbr)| \|f(\mbr)\|_{AP} \|\nabla f(\mbr)\|_{AP}.
%	\end{equation}
%	Finally, we set $ C = 2 \max |\nabla g(\mbr)| $, and divide by $ \|f(\mbr)\|_{AP} $.
%\end{proof}

\subsection{Incommensurate phase-field crystal (iPFC) model}
\label{subsec:iPFC}

%% Before introducing the iPFC model, we introduce the LP model firstly;
The simplest iPFC model may be the LP model which was originally
proposed to study the bi-frequency excited Faraday wave\,\cite{Lifshitz1997}. 
Concretely, the free energy functional of the LP model can be written as 
\begin{equation}
	F_{LP}[\psi(\mbr)] = \bbint \left\{ \frac{c}{2} \left[ (\Delta+1)
	(\Delta+q^{2}) \psi \right]^{2}
	+ \left( \frac{\varepsilon}{2} \psi^{2} - \frac{\alpha}{3} \psi^{3} +
	\frac{1}{4} \psi^{4} \right) \right\}\, d\mbr,
\end{equation}
where $q$ is an irrational number depending on the property of quasiperiodic structures.
The essential feature of the energy functional is the existence of two
characteristic length-scales, $1$ and $q$, which is a critical factor to
stabilize quasi-periodic structures. 
The LP model has been also used to study the soft-matter quasicrystals\,\cite{lifshitz2007}. 
However, this model could be able to globally stabilize the dodecagonal and decagonal rotational 
quasicrystals\,\cite{Jiang2015PRE}. 
%To generalize the scope of applications of iPFC model, Savitz et al. extended the 
%interaction potential from two characteristic length-scale potential to multiple characteristic
%length-scale. 
To generalize the iPFC model, Savitz et al. extended the interaction potential from two 
characteristic length-scales to multiple characteristic length-scales. 
In particular, the Lyapunov functional of an $m$-length-scale incommensurate system can be written as
\begin{equation}
	\begin{aligned}
		F[\psi(\mbr)] = \bbint \Bigg\{ \frac{c}{2} \Big[\prod_{j=1}^m
		(\Delta+q_j^2) \psi \Big]^2
		+ \Big(\frac{\varepsilon}{2}\psi^2 - \frac{\alpha}{3}\psi^3 +
		\frac{1}{4}\psi^4 \Big) \Bigg\} \,d\mbr,
	\end{aligned}
    \label{eq:energy.origin.L2}
\end{equation}
where $\psi(\mbr)$, $\mbr\in\mathbb{R}^d$ $(d=1,2,3)$, is the order parameter corresponding to the 
density profile of the system.
$q_{j}$ is the $j$-th characteristic length-scale which depends on the property of
aperiodic structures.
$c$ is a positive model parameter to ensure that the principle wavelengths are near the critical wavelengths. 
$\varepsilon$ and $\alpha$ are both model parameters related to physical conditions, such as 
temperature, pressure.
To reduce the number of model parameters, the iPFC energy functional can be rescaled by defining  
$F = c^{2} \mf$, $\psi(\mbr) = \sqrt{c} \phi(\mbr)$, $\tep = c \varepsilon$, and $\tal = \sqrt{c} \alpha$.
The rescaled functional is 
\begin{equation}
	\begin{aligned}
		\mf[\phi(\mbr)] = \bbint \Big\{ \frac{1}{2} [\mg\phi]^2 + \mn(\phi) \Big\} \,d\mbr
		 = \frac{1}{2} \|\mg\phi\|_{AP}^{2} + \left< \mn(\phi), 1 \right>_{AP},
	\end{aligned}
	\label{eq:energy.L2}
\end{equation}
where 
\begin{equation}
	\mg = \prod_{j=1}^m (\Delta+q_j^2),
	\quad \quad	
	\mn(\phi) = \frac{\tep}{2}\phi^2 - \frac{\tal}{3}\phi^3 + \frac{1}{4}\phi^4. 
\end{equation}

%% The general form of gradient flow;
To solve the iPFC free energy functional, we consider the following Allen-Cahn dynamic equation 
\begin{equation}
    \begin{aligned}
		\phi_{t} & = - \mw(\phi),
        \\
        \mw(\phi) & := \frac{\delta \mf}{\delta \phi} = \mg^2\phi + \mn'(\phi).
    \end{aligned}
    \label{eq:partial}
\end{equation}
The initial value is $\phi|_{t=0} = \phi_{0}$.
%% The energy dissipation mechanism; 
From Theorem \ref{thm:Green}, we can prove that the system \eqref{eq:partial}
satisfies the following energy dissipation law 
\begin{equation}
	\frac{d\mf(\phi)}{dt} 
	= \left< \frac{\delta \mf}{\delta\phi}, \phi_{t} \right>_{AP} 
	= - \norm{\mw(\phi)} \leq 0.
	\label{eq:energy}
\end{equation}
Then we impose the following mean zero constraint of order parameter on
the iPFC model to ensure the mass conservation
\begin{equation}
	\bbint \phi(\mbr)\, d\mbr = 0 .
	\label{eq:mass.conserve}
\end{equation}

\section{Numerical methods}
\label{sec:theory}

%% Introducing the theoretical framework;
In this section, we propose the second-order unconditionally energy stable
Crank-Nicolson scheme based on the SAV technique, and also give the error analysis.
Further, we use the SDC strategy to improve the temporal accuracy of the
second-order scheme. 

\subsection{The scalar auxiliary variable (SAV) approach}
\label{subsec:SAV}

%% Introduction of SAV;
%The IEQ method requires $ \mn(\phi) $ is bounded from below, and this may not hold for some physical models.
%Instead of assuming $ \mn(\phi) $ is uniformly bounded from below, the SAV approach only assumes a lower bounded bulk energy $ \left< \mn(\phi), 1 \right>_{AP} \geq -C_{1} $.
Let's suppose that $F_{1}(\phi) = \inner{\mn(\phi),1} + C_{1} \geq 0$, where $C_{1}$ is a positive constant. 
We introduce a scalar auxiliary variable $\mr = \sqrt{F_{1}(\phi)}$ to transform \eqref{eq:partial} 
into an equivalent system as 
\begin{subequations}
    \begin{align}
		\phi_{t} & = - \mw(\phi),
    	\label{eq:partial.SAV.a}
        \\
        \mw(\phi) & = \mg^2 \phi + \frac{\mr}{\sqrt{F_{1}(\phi)}} \mn'(\phi),
    	\label{eq:partial.SAV.b}
        \\
		\mr_{t} & = \inner{ \frac{\mn'(\phi)}{2\sqrt{F_{1}(\phi)}}, \phi_{t} }.
    	\label{eq:partial.SAV.c}
    \end{align}
    \label{eq:partial.SAV}
\end{subequations}
%By taking the almost periodic inner products of \eqref{eq:partial.SAV.a} and 
%\eqref{eq:partial.SAV.b} with $\mw$ and $-\phi_{t}$, multiplying \eqref{eq:partial.SAV.c} 
%with $2\mr$ and adding them together, we have the following energy dissipation property 
By taking the almost periodic inner products of \eqref{eq:partial.SAV.a} with $\mw$, and 
\eqref{eq:partial.SAV.b} with $-\phi_{t}$, multiplying \eqref{eq:partial.SAV.c} 
with $2\mr$ and adding them together, we have the following energy dissipation property 
\begin{equation}
	\frac{d}{dt} \left( \frac{1}{2} \norm{\mg\phi} + \mr^{2} - C_{1} \right) = -\norm{\mw} \leq 0.
	\label{eq:energy.SAV}
\end{equation}

%\textcolor{blue}{
%The first-order semi-discrete scheme is 
%\begin{subequations}
%	\begin{align}
%		\frac{\phi^{n+1}-\phi^{n}}{\tau^{n}} & = - \mw^{n+1},
%		\label{eq:SAV.1.a}
%		\\
%		\mw^{n+1} & = \mg^2\phi^{n+1} + \frac{\mr^{n+1}}{\sqrt{F_{1}(\phi^{n})}} \mn'(\phi^{n}),
%		\label{eq:SAV.1.b}
%		\\
%		\frac{\mr^{n+1}-\mr^{n}}{\tau^{n}} 
%		& = \inner{\frac{\mn'(\phi^{n})}{2\sqrt{F_{1}(\phi^{n})}}, \frac{\phi^{n+1}-\phi^{n}}{\tau^{n}} }. 
%		\label{eq:SAV.1.c}
%	\end{align}
%	\label{eq:SAV.1}
%\end{subequations}
%}
%Take the almost periodic inner products of \eqref{eq:SAV.1.a} and \eqref{eq:SAV.1.b} with 
%$\tau^{n}\mw^{n+1}$ and $-(\phi^{n+1}-\phi^{n})$, respectively. 
%Multiply \eqref{eq:SAV.1.c} with $2\mr^{n+1}$ and add them together. 
%Then the system \eqref{eq:SAV.1} supports the energy dissipation law with a modified energy 
%\begin{equation}
%	\widetilde{\mf}(\phi^{n},\mr^{n}) = \frac{1}{2} \|\mg\phi^{n}\|_{AP}^2 + (\mr^{n})^2 - C_1.
%\end{equation}
%
The SAV approach can construct high-order unconditionally energy stable schemes. 
In this section, we discuss a second-order semi-discrete scheme based on 
the Crank-Nicolson method.
Suppose the time interval $[0,T]$ is divided into $N_{T}$ non-overlapping subintervals by the partition 
$0=t^{0}<t^{1}<\cdots<t^{n}<\cdots<t^{N_{T}}=T$. 
The time step size is $\tau^{n} = t^{n+1} - t^{n}$. 
We denote $t^{n}+\tau^{n}/2$ as $t^{n+1/2}$. 
Let $\phi^{n} = \phi(t^{n})$ and $\mr^{n} = \mr(t^{n})$. 
%% The second-order SAV/CN scheme;
\begin{scheme}[SAV/CN]
	\label{scheme:SAV.CN}
	For $n \geq 1$, given $\phi^{n}$, $\mr^{n}$ and $\phi^{n-1}$, $\mr^{n-1}$, we update $\phi^{n+1}$ 
	and $\mr^{n+1}$ by 
	\begin{subequations}
		\begin{align}
		\frac{\phi^{n+1}-\phi^{n}}{\tau^{n}} & = - \mw^{n+1/2},
		\label{eq:SAV.CN.a}
		\\
		\mw^{n+1/2} & = \mg^2\phi^{n+1/2} + \frac{\mr^{n+1/2}}{\sqrt{F_{1}(\bphi^{n+1/2})}} \mn'(\bphi^{n+1/2}),
		\label{eq:SAV.CN.b}
		\\
		\frac{\mr^{n+1}-\mr^{n}}{\tau^{n}} 
		& = \inner{\frac{\mn'(\bphi^{n+1/2})}{2\sqrt{F_{1}(\bphi^{n+1/2})}}, \frac{\phi^{n+1}-\phi^{n}}{\tau^{n}} },
		\label{eq:SAV.CN.c}
		\end{align}
		\label{eq:SAV.CN}
	\end{subequations}
	where $\mr^{n+1/2} = (\mr^{n+1}+\mr^{n})/2$, $\phi^{n+1/2} = (\phi^{n+1}+\phi^{n})/2$ 
	and $\bphi^{n+1/2} = (3\phi^{n}-\phi^{n-1})/2$.
\end{scheme}

\begin{theorem}
	\label{thm:SAV.CN.energy}
	The SAV/CN scheme satisfies the following energy dissipation mechanism 
%	\textcolor{blue}{for any $\tau^{n}$}:
	\begin{equation}
		\mf_{SAV/CN}^{n+1} - \mf_{SAV/CN}^{n} \leq 0,
	\end{equation}
	where
	\begin{equation}
		\mf_{SAV/CN}^{n} = \frac{1}{2} \norm{\mg\phi^{n}} + (\mr^{n})^{2} - C_{1}.
	\end{equation}
\end{theorem}
\begin{proof}
	%We take the almost periodic inner products of \eqref{eq:SAV.CN.a} and \eqref{eq:SAV.CN.b} 
	%with $\tau^{n} \mw^{n+1/2}$ and $-(\phi^{n+1}-\phi^{n})$, respectively. 
	We take the almost periodic inner products of \eqref{eq:SAV.CN.a} with $\tau^{n}\mw^{n+1/2}$, 
	and \eqref{eq:SAV.CN.b} with $-(\phi^{n+1}-\phi^{n})$. 
	\begin{equation}
		\inner{\phi^{n+1}-\phi^{n},\mw^{n+1/2}} = -\norm{\mw^{n+1/2}},
		\label{eq:SAV.CN.a.law}
	\end{equation}
	\begin{equation}
		\begin{aligned}
			-\inner{\phi^{n+1}-\phi^{n},\mw^{n+1/2}} & = -\inner{\phi^{n+1}-\phi^{n},\mg^2\phi^{n+1/2}} 
			\\
			& -\inner{\phi^{n+1}-\phi^{n},\frac{\mr^{n+1/2}}{\sqrt{F_{1}(\bphi^{n+1/2})}} \mn'(\bphi^{n+1/2})}.
		\end{aligned}
		\label{eq:SAV.CN.b.law}
	\end{equation}
	Multiplying \eqref{eq:SAV.CN.c} with $2\mr^{n+1/2}$ yields 
	\begin{equation}
		\begin{aligned}
			2\mr^{n+1/2}(\mr^{n+1}-\mr^{n}) 
			= \inner{\frac{\mr^{n+1/2}}{\sqrt{F_{1}(\bphi^{n+1/2})}}\mn'(\bphi^{n+1/2}), \phi^{n+1}-\phi^{n} }. 
		\end{aligned}
		\label{eq:SAV.CN.c.law}
	\end{equation}
	Adding \eqref{eq:SAV.CN.a.law}, \eqref{eq:SAV.CN.b.law} and \eqref{eq:SAV.CN.c.law} together, we obtain 
	\begin{equation}
		2\mr^{n+1/2}(\mr^{n+1}-\mr^{n}) = -\inner{\phi^{n+1}-\phi^{n},\mg^2\phi^{n+1/2}} - \norm{\mw^{n+1/2}}. 
		\label{eq:SAV.CN.abc.law}
	\end{equation}
	Since $\mr^{n+1/2} = (\mr^{n+1}+\mr^{n})/2$ and $\phi^{n+1/2} = (\phi^{n+1}+\phi^{n})/2$, 
	\eqref{eq:SAV.CN.abc.law} can be recast as 
	\begin{equation}
		\frac{1}{2} \left( \norm{\mg\phi^{n+1}} - \norm{\mg\phi^{n}} \right) 
		+ (\mr^{n+1})^{2}-(\mr^{n})^{2} = -\norm{\mw^{n+1/2}} \leq 0. 
		\label{eq:SAV.CN.abc.law.re}
	\end{equation}
	The desired conclusion is obtained from the above equation. 
\end{proof}

\begin{remark}
	It is noted that the modified energy $\mf_{SAV/CN}^{n}$ is different from the
	original energy $\mf(\phi^{n})$ since $\mr^{n}$ is obtained from the iteration process.
\end{remark}

%% The implementation of SAV/CN;
\begin{remark}[The implementation of the SAV/CN scheme]
\label{remark:SAV.CN.implement}
	Denote 
	\begin{equation}
		u(t^{n+1/2}) =	\frac{\mn'(\phi(t^{n+1/2}))}{\sqrt{F_{1}(\phi(t^{n+1/2}))}}, 
		\quad \quad
		u^{n+1/2} =	\frac{\mn'(\bphi^{n+1/2})}{\sqrt{F_{1}(\bphi{n+1/2})}}.
        \label{eq:SAV.CN.note.b}
	\end{equation}
	Substituting \eqref{eq:SAV.CN.b} and \eqref{eq:SAV.CN.c} into
	\eqref{eq:SAV.CN.a}, we obtain
    \begin{equation}
        \frac{\phi^{n+1}-\phi^{n}}{\tau^{n}} = - \left[ \mg^2 \phi^{n+1/2} 
			+ u^{n+1/2} \left( \mr^{n} + \frac{1}{4} \inner{u^{n+1/2}, \phi^{n+1}-\phi^{n}} \right) \right].
		\label{eq:SAV.CN.phi.b}
    \end{equation}
	Eqn.\,\eqref{eq:SAV.CN.phi.b} can be rewritten as 
    \begin{equation}
    	\begin{aligned}
    		& \quad (I + \frac{\tau^{n}}{2}\mg^2) \phi^{n+1} 
				+ \frac{\tau^{n}}{4} u^{n+1/2} \inner{u^{n+1/2}, \phi^{n+1}} 
    		\\
    		& = (I - \frac{\tau^{n}}{2}\mg^2) \phi^{n} - \tau^{n} \mr^{n} u^{n+1/2} 
				+ \frac{\tau^{n}}{4} \inner{u^{n+1/2}, \phi^{n}} u^{n+1/2}.
    	\end{aligned}
        \label{eq:SAV.CN.mark.1}
    \end{equation}
	Taking the almost periodic inner product with $(I+\frac{1}{2}\tau^{n}\mg^2)^{-1} u^{n+1/2}$ leads to 
    \begin{equation}
        \inner{u^{n+1/2}, \phi^{n+1}} + \frac{\tau^{n}}{4} \gamma^{n} \inner{u^{n+1/2}, \phi^{n+1}}
        = \inner{u^{n+1/2}, (I + \frac{\tau^{n}}{2}\mg^2)^{-1} c^{n}},
        \label{eq:SAV.CN.uphi}
    \end{equation}
    where 
    \begin{align}
    	\gamma^{n} & = \inner{u^{n+1/2}, (I+\frac{\tau^{n}}{2}\mg^2)^{-1} u^{n+1/2}},
    	\\
    	c^{n} & = (I - \frac{\tau^{n}}{2}\mg^2) \phi^{n} - \tau^{n} \mr^{n} u^{n+1/2} 
			+ \frac{\tau^{n}}{4} \inner{u^{n+1/2}, \phi^{n}} u^{n+1/2}.
    \end{align}
	From \eqref{eq:SAV.CN.uphi}, we get 
    \begin{equation}
        \inner{u^{n+1/2}, \phi^{n+1}} 
			= \frac{\inner{u^{n+1/2}, (I + \frac{1}{2} \tau^{n} \mg^2)^{-1} c^{n}}}{I + \frac{1}{4} \tau^{n} \gamma^{n}}.
		\label{eq:SAV.CN.mark.2}
    \end{equation}
	Then we can directly calculate $\phi^{n+1}$ via \eqref{eq:SAV.CN.mark.1} and \eqref{eq:SAV.CN.mark.2}. 
\end{remark}

\subsection{Error estimate}
\label{subsec:error}

In this section, we will derive the error estimate of the SAV/CN scheme\,\ref{scheme:SAV.CN}.
Denote $e^{n} = \phi^{n} - \phi(t^{n})$, $w^{n+1} = \mw^{n+1} - \mw(t^{n+1})$,
and $r^{n} = \mr^{n} - \mr(t^{n})$, we have 
\begin{theorem}
	\label{thm:SAV.CN.AC.err}
	For the Allen-Cahn dynamic equation, assume that $\phi^{0}$ is almost periodic and 
	$\norm{ \phi_{t} }$ is bounded. 
	Considering a uniform time partition, \textit{i.e.}, $\tau=\tau^{k}, k \leq N_{T}$,  we have
	\begin{equation}
		\norm{\mg e^{k}} + (r^{k})^{2}
		\leq C \tau^{4} \int_{0}^{t^{k}} \left( \norm{\phi_{ttt}(s)} 
			+ |r_{ttt}(s)|^{2} \right) ds,
	\end{equation}
	where the constant $C$ is independent on $\tau$.
\end{theorem}
\begin{proof}
	Let's subtract \eqref{eq:partial.SAV} from \eqref{eq:SAV.CN} at $t^{n+1/2}$  
	\begin{equation}
		\begin{aligned}
			e^{n+1} - e^{n} = -\tau w^{n+1/2} + T_{1}^{n+1/2},
		\end{aligned}
		\label{eq:SAV.CN.CH.err.pf1.a}
	\end{equation}
	\begin{equation}
		\begin{aligned}
			w^{n+1/2} = \mg^{2} e^{n+1/2} + \mr^{n+1/2} u^{n+1/2} - \mr(t^{n+1/2}) u(t^{n+1/2}),
		\end{aligned}
		\label{eq:SAV.CN.CH.err.pf1.b}
	\end{equation}
	\begin{equation}
		\begin{aligned}
			r^{n+1} - r^{n} & = \frac{1}{2} \inner{ u^{n+1/2}, \phi^{n+1}-\phi^{n} } 
			\\
			& \quad - \frac{1}{2} \inner{ u(t^{n+1/2}), \tau \phi_{t}(t^{n+1/2}) } + T_{2}^{n+1/2},
		\end{aligned}
		\label{eq:SAV.CN.CH.err.pf1.c}
	\end{equation}
	where the truncation errors are given by
		\begin{align}
			T_{1}^{n+1/2} & = \tau \phi_{t}(t^{n+1/2}) - \left( \phi(t^{n+1}) - \phi(t^{n}) \right), 
			\\
			T_{2}^{n+1/2} & = \tau \mr_{t}(t^{n+1/2}) - \left( \mr(t^{n+1}) - \mr(t^{n}) \right). 
		\end{align}
	With the Taylor expansion, the truncation errors can be rewritten as 
		\begin{align}
			T_{1}^{n+1/2} & = \frac{1}{2} \int_{t^{n+1}}^{t^{n+1/2}} (t^{n+1}-s)^{2} \phi_{ttt}(s) ds
				- \frac{1}{2} \int_{t^{n}}^{t^{n+1/2}} (t^{n}-s)^{2} \phi_{ttt}(s) ds,
			\\
			T_{2}^{n+1/2} & = \frac{1}{2} \int_{t^{n+1}}^{t^{n+1/2}} (t^{n+1}-s)^{2} \mr_{ttt}(s) ds
				- \frac{1}{2} \int_{t^{n}}^{t^{n+1/2}} (t^{n}-s)^{2} \mr_{ttt}(s) ds. 
		\end{align}

	Firstly, making the almost periodic inner product of \eqref{eq:SAV.CN.CH.err.pf1.a} with $w^{n+1/2}$ yields 
	\begin{equation}
		\inner{e^{n+1}-e^{n}, w^{n+1/2}} + \tau \norm{w^{n+1/2}} = \inner{T_{1}^{n+1/2}, w^{n+1/2}}.
		\label{eq:SAV.CN.CH.err.pf2.a}
	\end{equation}
	Then its right-hand term can be bounded by
	\begin{equation}
		\begin{aligned}
			& \inner{T_{1}^{n+1/2}, w^{n+1/2}}
			\leq \frac{\tau}{2} \norm{w^{n+1/2}} + \frac{C}{\tau} \norm{T_{1}^{n+1/2}}
			\\
			& \leq \frac{\tau}{2} \norm{w^{n+1/2}} 
				+ C \tau^{4} \int_{t^{n}}^{t^{n+1}} \norm{\phi_{ttt}(s)} ds.
		\end{aligned}
		\label{eq:SAV.CN.CH.err.pf3.a}
	\end{equation}

	Secondly, by taking the almost periodic inner products of \eqref{eq:SAV.CN.CH.err.pf1.b} with $-(e^{n+1}-e^{n})$, 
	we obtain 
	\begin{equation}
		\begin{aligned}
			& -\inner{w^{n+1/2}, e^{n+1}-e^{n}} = -\frac{1}{2} \left( \norm{\mg e^{n+1}} - \norm{\mg e^{n}} \right)
			\\
			& \quad \quad \quad \quad \quad 
			- \inner{\mr^{n+1/2}u^{n+1/2} - \mr(t^{n+1/2})u(t^{n+1/2}), e^{n+1}-e^{n}}. 
		\end{aligned}
		\label{eq:SAV.CN.CH.err.pf2.b}
	\end{equation}
	Without the minus sign, the second term on the right-hand side in the above equation 
	can be transformed into 
	\begin{equation}
		\begin{aligned}
			& \quad \inner{\mr^{n+1/2}u^{n+1/2} - \mr(t^{n+1/2})u(t^{n+1/2}), e^{n+1}-e^{n}}
			\\
			& = r^{n+1/2} \inner{u^{n+1/2}, e^{n+1}-e^{n}}
			\\
			& \quad 
			+ \mr(t^{n+1/2}) \inner{u^{n+1/2} - u(t^{n+1/2}), e^{n+1}-e^{n}}.
		\end{aligned}
		\label{eq:SAV.CN.CH.err.pf3.b.0}
	\end{equation}
	Note that $|\mr(t)| \leq C$ and 
	\begin{equation}
		\begin{aligned}
			\| u^{n+1/2} - u(t^{n+1/2}) \|_{AP}
			& \leq C \| \mn'(\bphi^{n+1/2}) - \mn'(\phi(t^{n+1/2})) \|_{AP}
			\\
			& \leq C \left( \|e^{n}\|_{AP} + \|e^{n-1}\|_{AP} \right). 
		\end{aligned}
		\label{eq:SAV.CN.CH.err.pf3.b1}
	\end{equation}
	The last term on the right-hand side of \eqref{eq:SAV.CN.CH.err.pf3.b.0} can be estimated by  
	\begin{equation}
		\begin{aligned}
			& \quad \mr(t^{n+1/2}) \inner{u^{n+1/2} - u(t^{n+1/2}), e^{n+1}-e^{n}}
			\\
			& = \mr(t^{n+1/2}) \inner{u^{n+1/2} - u(t^{n+1/2}), -\tau w^{n+1/2}+T_{1}^{n+1/2}}
			\\
			& \leq \frac{\tau}{2} \norm{w^{n+1/2}} + C \tau \norm{u^{n+1/2} - u(t^{n+1/2})}
			+ \frac{C}{\tau} \norm{T_{1}^{n+1/2}}
			\\
			& \leq \frac{\tau}{2} \norm{w^{n+1/2}} 
			+ C \tau \left( \norm{e^{n}} + \norm{e^{n-1}} \right)
			\\
			& \quad + C \tau^{4} \int_{t^{n}}^{t^{n+1}} \norm{\phi_{ttt}(s)} ds.
		\end{aligned}
		\label{eq:SAV.CN.CH.err.pf3.b}
	\end{equation}

	Thirdly, multiplying the both sides of \eqref{eq:SAV.CN.CH.err.pf1.c} by $2r^{n+1/2}$, we obtain 
	\begin{equation}
		\begin{aligned}
			& (r^{n+1})^{2} - (r^{n})^{2}
			= r^{n+1/2} \inner{u^{n+1/2}, \phi^{n+1}-\phi^{n}}
			\\
			& \quad \quad \quad \quad \quad 
			- r^{n+1/2} \inner{u(t^{n+1/2}), \tau \phi_{t}(t^{n+1/2})} + 2r^{n+1/2} T_{2}^{n+1/2}. 
		\end{aligned}
		\label{eq:SAV.CN.CH.err.pf2.c}
	\end{equation}
	The first two terms on the right-hand side of \eqref{eq:SAV.CN.CH.err.pf2.c} can be rewritten as 
	\begin{equation}
		\begin{aligned}
			& \quad r^{n+1/2} \inner{u^{n+1/2}, \phi^{n+1}-\phi^{n}}
			- r^{n+1/2} \inner{u(t^{n+1/2}), \tau \phi_{t}(t^{n+1/2})}
			\\
			& = r^{n+1/2} \inner{u^{n+1/2}, e^{n+1}-e^{n}}
			- r^{n+1/2} \inner{u(t^{n+1/2}), T_{1}^{n+1/2}}
			\\
			& \quad + r^{n+1/2} \inner{ u^{n+1/2} - u(t^{n+1/2}), \phi(t^{n+1})-\phi(t^{n}) }, 
		\end{aligned}
		\label{eq:SAV.CN.CH.err.pf3.c}
	\end{equation}
	where the last two terms on the right-hand side satisfy  
	\begin{equation}
		\begin{aligned}
			- r^{n+1/2} \inner{u(t^{n+1/2}), T_{1}^{n+1/2}}
			& \leq C \tau \left( (r^{n+1})^{2} + (r^{n})^{2} \right)  
			\\
			& \quad + C \tau^{4} \int_{t^{n}}^{t^{n+1}} \norm{\phi_{ttt}(s)} ds,
		\end{aligned}
	\end{equation}
	and
	\begin{equation}
		\begin{aligned}
			& \quad r^{n+1/2} \inner{ u^{n+1/2} - u(t^{n+1/2}), \phi(t^{n+1})-\phi(t^{n}) }
			\\
			& \leq C \tau \norm{\phi_{t}} \left( (r^{n+1/2})^{2} 
			+ \norm{ u^{n+1/2} - u(t^{n+1/2}) } \right)
			\\
			& \leq C \tau \left( (r^{n+1})^{2} + (r^{n})^{2} + \norm{e^{n}} + \norm{e^{n-1}} \right).
		\end{aligned}
	\end{equation}
	Therefore the last term on the right-hand side of \eqref{eq:SAV.CN.CH.err.pf2.c}
	can be bounded by 
	\begin{equation}
		\begin{aligned}
			2r^{n+1/2} T_{2}^{n+1/2} \leq C \tau \left( (r^{n+1})^{2} + (r^{n})^{2} \right) 
				+ C \tau^{4} \int_{t^{n}}^{t^{n+1}} \left| r_{ttt}(s) \right|^{2} ds.
		\end{aligned}
	\end{equation}
%	
%	With Theorem\,\ref{thm:Green} and \ref{thm:Poincare}, we have 
%	\begin{equation}
%		\begin{aligned}
%			\norm{(\Delta+q_{1}^{2}) e^{n}}
%			& \leq C \norm{ \Delta (\Delta+q_{1}^{2}) e^{n} }
%			\\
%			& \leq C \norm{ (\Delta+q_{1}^{2})(\Delta+q_{2}^{2}) e^{n}} 
%			+ C q_{2}^{2} \norm{ (\Delta+q_{1}^{2}) e^{n}}.
%		\end{aligned}
%	\end{equation}
%	It indicates that
%	\begin{equation}
%		\begin{aligned}
%			\norm{\nabla e^{n}} &\leq C \left( \norm{e^{n}} + \norm{\Delta e^{n}} \right)
%			\leq C \left( \norm{e^{n}} + \norm{\mg e^{n}} \right),
%		\end{aligned}
%	\end{equation}
	With the above estimates, adding \eqref{eq:SAV.CN.CH.err.pf2.a}, \eqref{eq:SAV.CN.CH.err.pf2.b} 
	and \eqref{eq:SAV.CN.CH.err.pf2.c} together leads to 
	\begin{equation}
		\begin{aligned}
			& \quad \frac{1}{2} \left( \norm{\mg e^{n+1}} - \norm{\mg e^{n}} \right) + (r^{n+1})^{2} - (r^{n})^{2}
			\\
			& \leq C \tau \left( \norm{e^{n}} + \norm{e^{n-1}} + (r^{n+1})^{2} + (r^{n})^{2} \right)
			\\
			& \quad + C \tau^{4} \int_{t^{n}}^{t^{n+1}} \norm{\phi_{ttt}(s)} ds
			+ C \tau^{4} \int_{t^{n}}^{t^{n+1}} \left| r_{ttt}(s) \right|^{2} ds.
		\end{aligned}
	\end{equation}
%	By direct calculation, 
%	\begin{equation}
%		\begin{aligned}
%			r_{ttt} & = -\frac{\inner{\mn'(\phi),\phi_{t}}}{2\sqrt{F_{1}^{3}}} \left( \inner{\mn^{''}(\phi),(\phi_{t})^{2}} + \inner{\mn'(\phi),\phi_{tt}} \right)
%			+ \frac{3}{8\sqrt{F_{1}^{5}}} \left( \inner{\mn'(\phi),\phi_{t}} \right)^{3}
%			\\
%			& \quad - \frac{1}{4\sqrt{F_{1}^{3}}} \inner{\mn'(\phi),\phi_{t}} \left( \inner{\mn^{''},(\phi_{t})^{2}} + \inner{\mn'(\phi),\phi_{tt}} \right)
%			\\
%			& \quad + \frac{1}{2\sqrt{F_{1}}} \left( \inner{\mn^{'''},(\phi_{t})^{3}} + \inner{\mn^{''}(\phi),3\phi_{t}\phi_{tt}} + \inner{\mn^{'}(\phi),\phi_{ttt}} \right).
%		\end{aligned}
%	\end{equation}
	Then we obtain the desired conclusion by summing over $n$, $n =
	0,1,\cdots,k-1$, and using the Gronwall inequality. 
%	\begin{equation}
%		\begin{aligned}
%			\norm{\mg e^{k}} + (r^{k})^{2}
%			\leq C \tau^{4} \int_{0}^{t^{k}} \left( \norm{\phi_{ttt}(s)} + |r_{ttt}(s)|^{2} \right) ds.
%		\end{aligned}
%	\end{equation}
\end{proof}

%% Introducing the spectral deferred correction method to improve the accuracy;
%% How to implement based on IEQ and SAV approaches;
\subsection{Spectral deferred correction (SDC) approach}
\label{subsec:SDC}

%The SDC approach is a special case of the deferred correction method. 
%This approach has special numerical form depending on the chosen numerical schemes. 
%And it requires special time nodes to improve the accuracy of integral. 
%This approach can be constructed easily and systematically since it is a one step method. 
The SDC approach\,\cite{Dutt2000, Xu2019} is an efficient method to improve the numerical precision for an 
existing scheme using spectral collocation points. 
Firstly, we introduce the basic idea of the SDC method. 
Integrating both sides of the Eqn.\,\eqref{eq:partial} with respect to $t$, we have 
\begin{equation}
	\phi(t) = \phi(0) - \int_{0}^{t} \mw(\phi(\tau)) \,d\tau.
	\label{eq:SDC.integrate}
\end{equation}
Suppose the approximation solution $\phi_{[0]}(t)$ has been calculated by some
numerical schemes. 
Then we define the residual $R_{[0]}(t)$ and the error $\epsilon_{[0]}(t)$ as 
\begin{equation}
	R_{[0]}(t) = \phi(0) - \int_{0}^{t} \mw(\phi_{[0]}(\tau)) \,d\tau - \phi_{[0]}(t), 
	\label{eq:SDC.residual}
\end{equation}
\begin{equation}
	\epsilon_{[0]}(t) = \phi(t) - \phi_{[0]}(t). 
	\label{eq:SDC.error}
\end{equation}
Replacing \eqref{eq:SDC.error} into \eqref{eq:SDC.integrate} yields
\begin{equation}
	\phi(t) = \phi(0) - \int_{0}^{t} \mw(\phi_{[0]}(\tau) + \epsilon_{[0]}(\tau)) \,d\tau.
	\label{eq:SDC.integrate.phi}
\end{equation}
Then we insert \eqref{eq:SDC.integrate.phi} into \eqref{eq:SDC.error} and subtract \eqref{eq:SDC.residual} 
\begin{equation}
	\epsilon_{[0]}(t) - R_{[0]}(t) = - \int_{0}^{t} \mw(\phi_{[0]}(\tau) 
		+ \epsilon_{[0]}(\tau)) d\tau + \int_{0}^{t} \mw(\phi_{[0]}(\tau)) d\tau.
	\label{eq:SDC.error.residual}
\end{equation}
By taking the derivative of both sides of the above equation, we obtain 
\begin{equation}
	\frac{d\epsilon_{[0]}(t)}{dt} = - \mw(\phi_{[0]}(t) + \epsilon_{[0]}(t)) + \mw(\phi_{[0]}(t)) + \frac{dR_{[0]}(t)}{dt}.
	\label{eq:SDC.derivative}
\end{equation}

%% The discrete numerical schemes of the SDC approach;
Then, we apply the SDC approach into the second-order SAV/CN scheme to improve the
numerical accuracy.  We denote this strategy as SAV/CN+SDC. 
%\begin{scheme}[SAV/CN+SDC]
%\label{scheme.SAV.CN.SDC}
To calculate the integral precisely, we adopt the following Chebyshev nodes, 
\begin{equation}
	t^{n} = \frac{T}{2} - \frac{T}{2} \cos\left( \frac{n\pi}{N_{T}} \right),
	~~ n = 0, 1, \cdots, N_{T}.
	\label{eq:SDC.tn}
\end{equation}
The time step size is $\tau^{n} = t^{n+1} - t^{n}$.
To obtain the approximation solution $\phi_{[0]}(t)$, we rewrite the SAV/CN scheme as
\begin{align}
	\frac{\phi_{[0]}^{n+1}-\phi_{[0]}^{n}}{\tau^{n}} & = - \mw_{[0]}^{n+1/2},
	\nonumber
	\\
	\mw_{[0]}^{n+1/2} & = \mg^2\phi_{[0]}^{n+1/2} 
		+ \frac{\mr_{[0]}^{n+1/2}}{\sqrt{F_{1}(\bphi_{[0]}^{n+1/2})}} \mn'(\bphi_{[0]}^{n+1/2}),		
	\label{eq:SDC.SAV.CN.phi.b}
	\\
	\frac{\mr_{[0]}^{n+1}-\mr_{[0]}^{n}}{\tau^{n}} 
	& = \inner{\frac{\mn'(\bphi_{[0]}^{n+1/2})}{2\sqrt{F_{1}(\bphi_{[0]}^{n+1/2})}}, 
		\frac{\phi_{[0]}^{n+1}-\phi_{[0]}^{n}}{\tau^{n}} },
	\nonumber
\end{align}
where $ \phi_{[0]}^{0} = \phi^{0} $, $\mr_{[0]}^{n+1/2} = (\mr_{[0]}^{n+1} + \mr_{[0]}^{n})/2$, 
$\phi_{[0]}^{n+1/2} = (\phi_{[0]}^{n+1} + \phi_{[0]}^{n})/2$ and 
$\bphi_{[0]}^{n+1/2} = (3\phi_{[0]}^{n} - \phi_{[0]}^{n-1})/2$.
We adopt a similar strategy to discretize \eqref{eq:SDC.derivative} as follows 
\begin{subequations}
	\begin{align}
		\frac{\epsilon_{[0]}^{n+1}-\epsilon_{[0]}^{n}}{\tau^{n}} 
		& = - \mw_{[0]}^{n+1/2,\epsilon} + \mw_{[0]}^{n+1/2}
		+ \frac{R_{[0]}^{n+1} - R_{[0]}^{n}}{\tau^{n}},
		\label{eq:SDC.SAV.CN.epsilon.a}
		\\
		\mw_{[0]}^{n+1/2,\epsilon} & = \mg^2 \phi_{[0]}^{n+1/2,\epsilon} 
		+ \frac{\mr_{[0]}^{n+1/2}}{\sqrt{F_{1}(\bphi_{[0]}^{n+1/2})}} \mn'(\bphi_{[0]}^{n+1/2,\epsilon}),
		\label{eq:SDC.SAV.CN.epsilon.b}
%		\mw_{[0]}^{n+1/2,\epsilon} & = \mg^2 \phi_{[0]}^{n+1/2,\epsilon} 
%		+ \frac{\mr_{[0]}^{n+1,\epsilon} + \mr_{[0]}^{n,\epsilon}}{2\sqrt{F_{1}(\bphi_{[0]}^{n+1/2,\epsilon})}} 
%		\mn'(\bphi_{[0]}^{n+1/2,\epsilon}),
%		\\
%		\frac{\mr_{[0]}^{n+1,\epsilon}-\mr_{[0]}^{n,\epsilon}}{\tau^{n}} 
%		& = \inner{\frac{\mn'(\bphi_{[0]}^{n+1/2,\epsilon})}{2\sqrt{F_{1}(\bphi_{[0]}^{n+1/2,\epsilon})}}, 
%			\frac{\phi_{[0]}^{n+1,\epsilon}-\phi_{[0]}^{n,\epsilon}}{\tau^{n}} },
	\end{align}
	\label{eq:SDC.SAV.CN.epsilon}
\end{subequations}
where $\epsilon_{[0]}^{0} = \phi(0) - \phi_{[0]}^{0} = 0$, 
$\phi_{[0]}^{n+1/2,\epsilon} = \phi_{[0]}^{n+1/2} + \epsilon_{[0]}^{n+1/2}$, 
and $\bphi_{[0]}^{n+1/2,\epsilon} = \bphi_{[0]}^{n+1/2} + \bar{\epsilon}_{[0]}^{n+1/2}$. 
Substituting \eqref{eq:SDC.SAV.CN.phi.b} and \eqref{eq:SDC.SAV.CN.epsilon.b} into \eqref{eq:SDC.SAV.CN.epsilon.a} 
and eliminating the residual by \eqref{eq:SDC.residual}, we obtain the following linear solvable equation
\begin{equation}
	\begin{aligned}
		& \left( I + \frac{\tau^{n}}{2} \mg^{2} \right) \epsilon_{[0]}^{n+1} 
		= \left( I - \frac{\tau^{n}}{2} \mg^{2} \right) \epsilon_{[0]}^{n} 
		- \int_{t^{n}}^{t^{n+1}} \mw(\phi_{[0]}(\tau^{n})) d\tau^{n}   
		\\
		& \quad - \phi_{[0]}^{n+1} + \phi_{[0]}^{n} 
		- \tau^{n} \frac{\mr_{[0]}^{n+1/2}}{\sqrt{F_{1}(\bphi_{[0]}^{n+1/2})}} 
		\left\{ \mn'(\bphi_{[0]}^{n+1/2,\epsilon}) - \mn'(\bphi_{[0]}^{n+1/2}) \right\}.
	\end{aligned}
	\label{eq:SDC.CN.epsilon.update}
\end{equation}
A more accurate solution can be updated by 
\begin{equation}
	\phi_{[1]}^{n+1} = \phi_{[0]}^{n+1} + \epsilon_{[0]}^{n+1}.
	\label{eq:SDC.CN.phi.update}
\end{equation}
%\end{scheme}

\subsection{The projection method (PM) discretization}
\label{subsec:PM}

%% A brief introduction of PM;
The PM is an accurate approach in computing aperiodic structures that can 
avoid the Diophantine approximation error.
%Due to aperiodic structures lose periodicity in at least one direction, they cannot be calculated 
%within a periodic region. 
%The CAM is a common approach which approximates an incommensurate structure by a big crystal phase but 
%brings the Diophantine approximation error.
The PM is based on the fact that the $d$-dimensional aperiodic structure can be embedded into an 
$n$-dimensional periodic structure $(n\geq d)$. 
The dimensionality $n$ is determined by the spectrum structures of the aperiodic
structures. In particular, $n$ is the number of linearly independent vectors
over the rational number field which span the spectrum.  
In the PM, the $d$-dimensional order parameter $\phi(\mbr)$ can be expanded as 
\begin{equation}
    \phi(\mbr) = \sum_{\mbh\in\mathbb{Z}^{n}} \hphi(\mbh)
    e^{i[(\mathcal{P} \cdot \mathbf{Bh})^T \cdot \mbr]}, ~~~
	\mbr\in\mathbb{R}^d.
	\label{eq:Bohr.Fourier.4D}
\end{equation}
where $\mathbf{B} \in \mathbb{R}^{n\times n}$ is an invertible matrix related to
the $n$-dimensional primitive reciprocal lattice. 
The projection matrix $\mathcal{P} \in \mathbb{R}^{d\times n}$ depends on the property 
of aperiodic structures. 
If we consider $d$-dimensional periodic structures, the projection matrix degenerates 
to a $d$-order identity matrix. 
Therefore, the PM provides a unified framework in calculating periodic and aperiodic crystals.
More details about the PM can refer to \cite{Jiang2014}.
The Fourier coefficient $\hphi(\mbh)$ satisfies 
\begin{equation}
	X := \left\{(\hphi(\mbh))_{\mbh\in\mathbb{Z}^n}:
	\hphi(\mbh)\in\mathbb{C}, ~
	\sum_{\mbh\in\mathbb{Z}^n}|\hphi(\mbh)|<\infty \right\}.
\end{equation}
In practice, let $\mathbf{N}=(N_1, N_2, \dots, N_n)\in \mathbb{N}^n$, and 
\begin{equation} 
	X_{\mathbf{N}} := \{\hphi(\mbh)\in X: \hphi(\mbh) = 0, ~\mbox{for all}~ |h_j|> N_j/2, ~ j=1,2,\dots,n \}.  
\end{equation}
The number of elements in the set is $N=(N_1+1)(N_2+1) \cdots (N_n+1)$. 
%Together with \eqref{eq:Bohr.Fourier} and \eqref{eq:Bohr.Fourier.4D}, the discretized energy function is 
%\begin{equation}
%    \begin{aligned}
%        \mf_{\mbh}[\hat{\Phi}] & = \frac{1}{2} \sum_{\mbh_{1}+\mbh_{2}=\bm{0}} 
%            \prod_{j=1}^m [q_j^2 - (\mathcal{P}\mathbf{Bh})^{T} (\mathcal{P}\mathbf{Bh}) ]^{2} 
%            \hphi(\mbh_{1}) \hphi(\mbh_{2})
%            \\
%            & + \frac{\tep}{2} \sum_{\mbh_{1}+\mbh_{2}=\bm{0}} \hphi(\mbh_{1}) \hphi(\mbh_{2})
%            - \frac{\tal}{3} \sum_{\mbh_{1}+\mbh_{2}+\mbh_{3}=\bm{0}} \hphi(\mbh_{1}) 
%            \hphi(\mbh_{2}) \hphi(\mbh_{3})
%            \\
%            & \quad + \frac{1}{4} \sum_{\mbh_{1}+\mbh_{2}+\mbh_{3}+\mbh_{4}=\bm{0}} \hphi(\mbh_{1}) 
%            \hphi(\mbh_{2}) \hphi(\mbh_{3}) \hphi(\mbh_{4}),
%    \end{aligned}	
%\end{equation}
%where $\mbh_{j}\in\mathbb{Z}^{n}$, $\hphi_{j}\in X_{\mathbf{N}}$, $j=1,2,3,4$, 
%$\hat{\Phi} = (\hphi_{1}, \hphi_{2}, \cdots, \hphi_{N}) \in \mathbb{C}^{N}$. 
Using the PM, the SAV/CN scheme of full discretization reads 
\begin{equation}
	\begin{aligned}
		\hphi^{n+1}(\mbh)-\hphi^{n}(\mbh) & = - \tau \widehat{\mw}^{nh}(\mbh),
		\\
		\widehat{\mw}^{nh}(\mbh) & = \frac{1}{2} 
			\prod_{j=1}^m [q_j^2 - (\mathcal{P}\mathbf{Bh})^{T} (\mathcal{P}\mathbf{Bh}) ]^{2} 
			[\hphi^{n+1}(\mbh)+\hphi^{n}(\mbh)] 
			\\
			& \quad + \frac{\mr^{n+1}+\mr^{n}}{2\sqrt{F_{1}^{nh}[\hat{\Phi}]}} \widehat{\mn'}^{nh}(\mbh),
		\\
		\mr^{n+1}-\mr^{n} & = 
			\sum_{\mbh_{1}+\mbh_{2}=\bm{0}} \frac{\widehat{\mn'}^{nh}(\mbh_{1})}{2\sqrt{F_{1}^{nh}[\hat{\Phi}]}} 
			[\hphi^{n+1}(\mbh_{2})-\hphi^{n}(\mbh_{2})],
	\end{aligned}
	\label{eq:SAV.CN.dis}
\end{equation}
where 
\begin{equation}
	\begin{aligned}
		\widehat{\mn'}^{nh}(\mbh) & = \tep\hphi^{nh}(\mbh) 
			- \tal \sum_{\mbh_{1}+\mbh_{2}=\mbh} \hphi^{nh}(\mbh_{1}) \hphi^{nh}(\mbh_{2})
			\\
			& \quad + \sum_{\mbh_{1}+\mbh_{2}+\mbh_{3}=\mbh} \hphi^{nh}(\mbh_{1}) 
			\hphi^{nh}(\mbh_{2}) \hphi^{nh}(\mbh_{3}),
		\\
		F_{1}^{nh}[\hat{\Phi}]
			& = \frac{\tep}{2} \sum_{\mbh_{1}+\mbh_{2}=\bm{0}} \hphi^{nh}(\mbh_{1}) \hphi^{nh}(\mbh_{2})
			- \frac{\tal}{3} \sum_{\mbh_{1}+\mbh_{2}+\mbh_{3}=\bm{0}} \hphi^{nh}(\mbh_{1}) 
			\hphi^{nh}(\mbh_{2}) \hphi^{nh}(\mbh_{3})
		\\
			& \quad + \frac{1}{4} \sum_{\mbh_{1}+\mbh_{2}+\mbh_{3}+\mbh_{4}=\bm{0}} \hphi^{nh}(\mbh_{1}) 
			\hphi^{nh}(\mbh_{2}) \hphi^{nh}(\mbh_{3}) \hphi^{nh}(\mbh_{4}) + C_{1},
		\\
		\hphi^{nh}(\mbh) & = \frac{3\hphi^{n+1}(\mbh) - \hphi^{n}(\mbh)}{2}. 
	\end{aligned}
\end{equation}
In the above equations, the nonlinear terms are $n$-dimensional convolutions in the Fourier space.
Directly computing them is extremely expensive. 
To avoid this, we use the pseudospectral method through the $n$-dimensional fast Fourier
transformation to compute them efficiently in the $n$-dimensional time domain by simple multiplication.
The mass conservation constraint \eqref{eq:mass.conserve} can be satisfied through 
\begin{equation}
	e_{1}^{T} \hat{\Phi} = 0,
\end{equation}
where $e_{1} = (1,0,\cdots,0)^{T} \in \mathbb{R}^{N}$.

\section{Numerical results}
\label{sec:results}

%% Brief introduction of this section;
In this section, we present several numerical examples to verify the accuracy 
of the SAV/CN and SAV/CN+SDC schemes and to illustrate the advantages of the
higher-order scheme in dynamic evolution.  We also show the influence of
multiple length-scales on the thermodynamic stability of aperiodic structures. 

\subsection{Accuracy}
\label{subsec:results.efficiency}

%% The time numerical convergence rate;
In this subsection, we take the two characteristic length scale iPFC model in one-dimensional 
space to test the numerical accuracy of the SAV/CN and SAV/CN+SDC schemes.
The model parameters are set as $q_{1} = \sqrt{2}$, $q_{2} = \sqrt{3}$, $\tep = 10$ and $\tal = 4$. 
The computational domain is $[0,2\pi]$.
Correspondingly, the projection matrix $\mathcal{P}$ and $\mathbf{B}$ in the PM both
are $1$. The initial data is chosen as $\phi(x,0) = \sin(x)$. 
We check the temporal accuracy by taking the space discretization $N=128$. 
The numerical solution of $N_{T}=2048$ is set as the reference solution.
Tab.\,\ref{tab:rate.SDC} shows the errors and convergence rates of the SAV/CN and SAV/CN+SDC schemes at $T=0.2$.
One can observe that the numerical accuracy of the SAV/CN scheme is second-order and
can be improved to fourth-order by the SDC approach. 
\begin{table}[htbp]
    \centering
    \caption{Errors and convergence rates of the SAV/CN and SAV/CN+SDC schemes for the Allen-Cahn equation. 
		The numerical solution of $N_{T} = 2048$ is regarded as the reference value.}
	\label{tab:rate.SDC}
    \begin{tabular}{|*{6}{c|}}
        \hline
        \diagbox[width=9em]{Scheme}{$N_{T}$} & & 64 & 128 & 256 & 512 \\
        \hline
        \multirow{2}*{SAV/CN} & Error & 4.75E-3 & 1.17E-3 & 2.91E-4 & 7.17E-5  \\
        \cline{2-6}
        ~ & Rate & - & 2.01 & 2.02 & 2.07 \\
        \hline
        \multirow{2}*{SAV/CN + SDC} & Error & 1.16E-5 & 6.78E-7 & 4.04E-8 & 2.46E-9  \\
        \cline{2-6}
        ~ & Rate & - & 4.07 & 4.03 & 4.02 \\
        \hline
    \end{tabular}
\end{table}

\subsection{Dynamic evolution}
\label{subsec:evolution}

%% The SDC approach improves the numerical convergence rate and further efficiently decreases the number 
%% of space discrete points;
In this subsection, we simulate the dynamic process of $2$-dimensional
dodecagonal quasicrystal (DDQC) using the iPFC model with two length-scales. 
The model parameters are $q_{1} = 1$, $q_{2} = 2\cos(\pi/12)$, $\tep = -2$ and
$\tal = 2$. 
The initial value is the DDQC whose spectral distribution and real morphology 
are shown in Fig.\,\ref{fig:DDQC.Ini}. 
The big blue dot represents the origin and the others are the $24$ basic Fourier modes located on 
the circles of radii $q_{1}$ and $q_{2}$, respectively. 
\begin{figure}[htbp]
	\centering
	\includegraphics[scale=0.15]{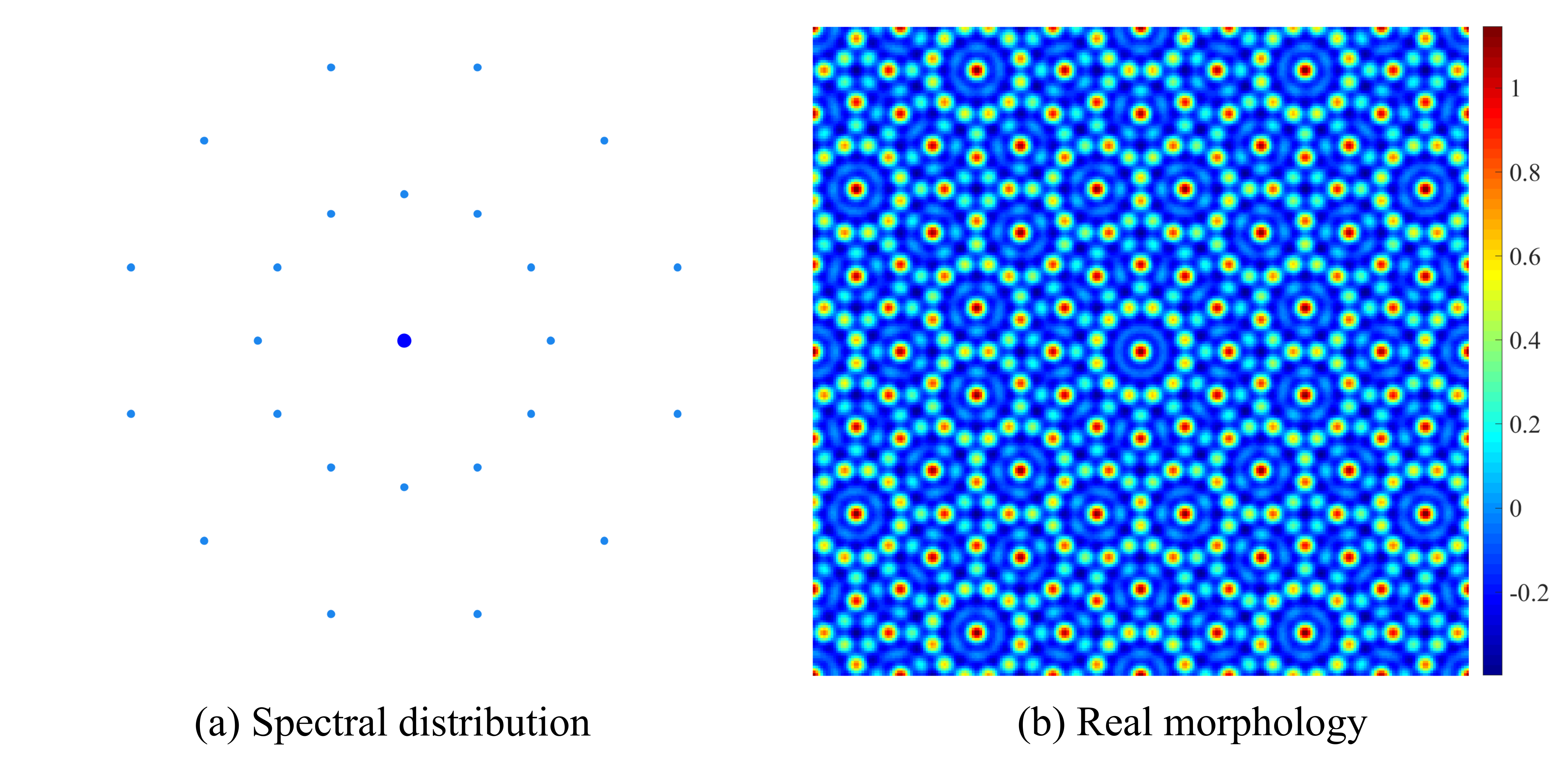}
	\caption{The spectral distribution and the corresponding real morphology of the initial value.}
	\label{fig:DDQC.Ini}
\end{figure}
When calculating the DDQC, we adopt the PM to discretize the spatial functions in 
$4$-dimensional space with $24^4$ trigonometric functions.
The $24^4$ basis functions bring a negligible spatial error comparing the temporal error. 
%\textcolor{blue}{
The projection matrix in the PM is 
\begin{equation}
	\mathcal{P} = \left( 
		\begin{array}{cccc} 
			1 & \cos(\pi/6) & \cos(\pi/3) & 0 \\ 
			0 & \sin(\pi/6) & \sin(\pi/3) & 1 
		\end{array} \right), 
\end{equation}
and the $\mathbf{B}$ is a 4-order identity matrix. 
%}

%% The pivotal morphologies of the dynamic simulation process;
We use the second-order SAV/CN scheme with $N_{T} = 256$ to simulate 
the dynamic evolution for the DDQC. 
Fig.\,\ref{fig:dynamic} shows the change tendency of the energy value. 
The corresponding morphologies at $t = 0, 50, 100, 150, 200$ are presented in
Fig.\,\ref{fig:dynamic.moph}.
As one can see, the proposed scheme satisfies the energy dissipation law. 
%\textcolor{blue}{
It should be noted that the value $C_1$ in the SAV approach plays an important
role in computational simulations. In our computation, the $C_1$ chosen as $10^{16}$
which maintains the consistency of original and modified energy values.
%}
\begin{figure}[htbp]
    \centering
    \includegraphics[scale=0.20]{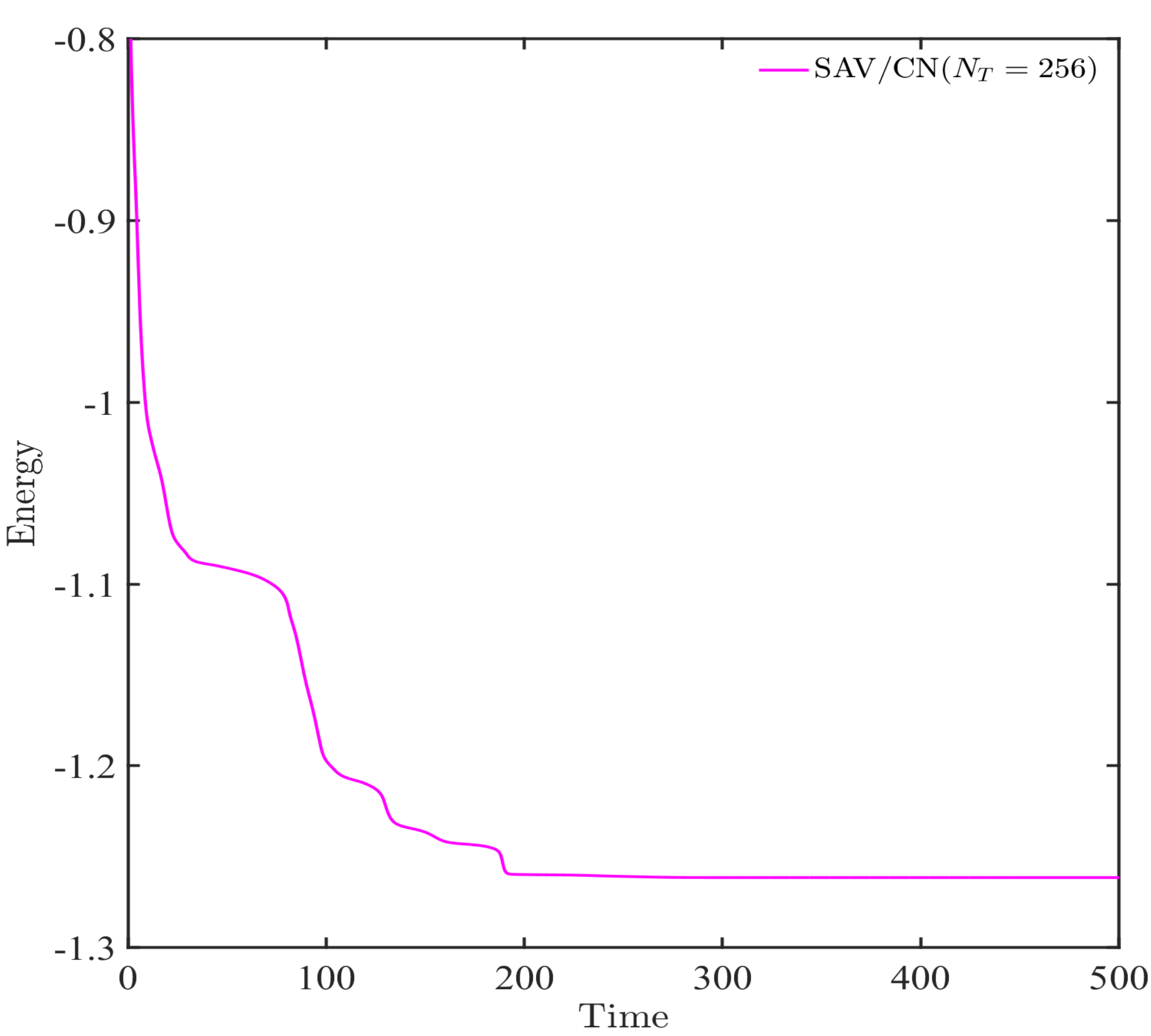}
    \caption{The time evolution of energy which is obtained by the SAV/CN scheme in the case of $N_{T}=256$.
		The model parameters are $q_{1} = 1$, $q_{2} = 2\cos(\pi/12)$, $\tep = -2$ and $\tal = 2$.}
	\label{fig:dynamic}
\end{figure}
\begin{figure}[htbp]
	\centering
	\includegraphics[scale=0.15]{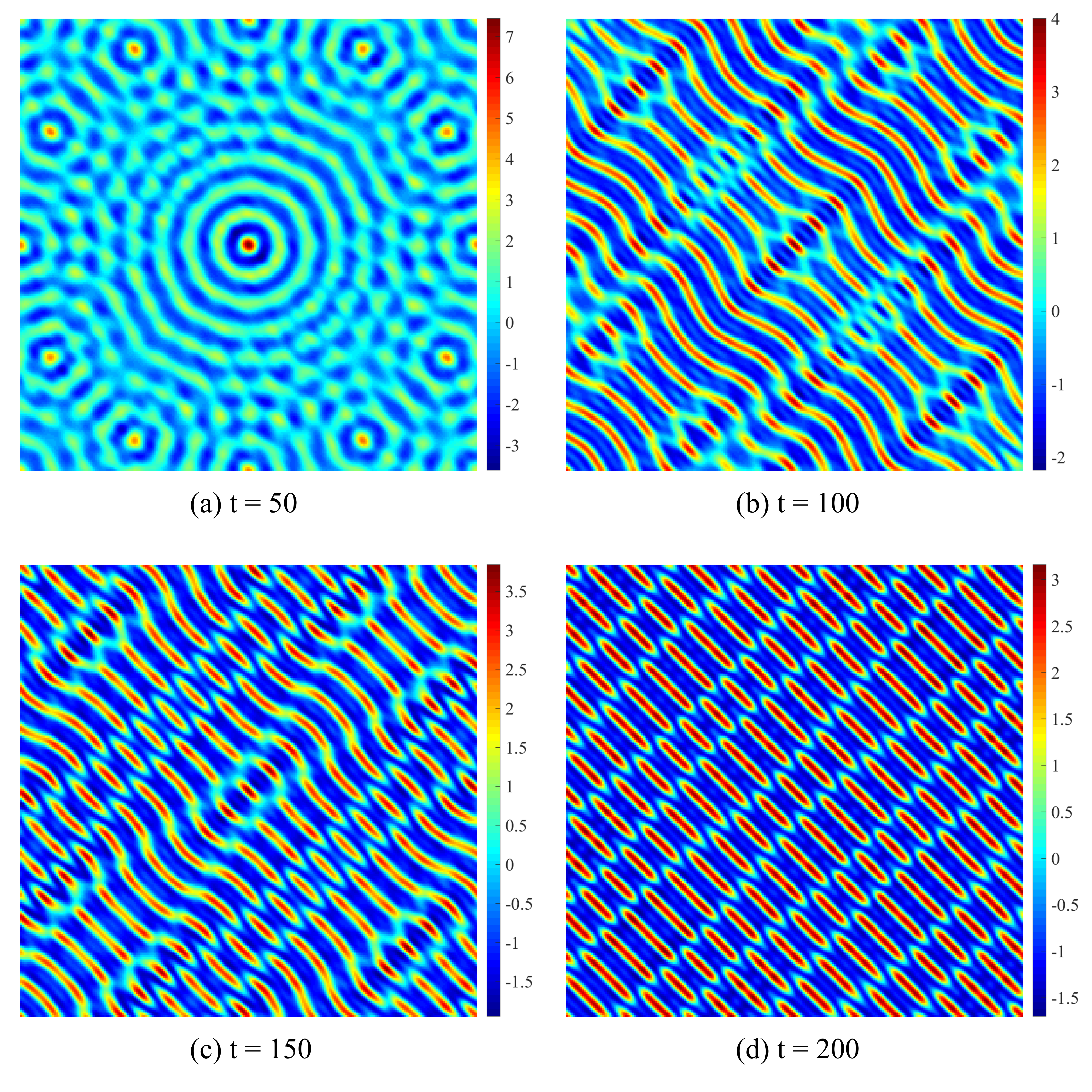}
	\caption{The morphologies of dynamic evolution in Fig.\,\ref{fig:dynamic}.
		Snapshots are taken at $t = 50, ~ 100, ~ 150, ~ 200$, respectively.}
	\label{fig:dynamic.moph}
\end{figure}

To show the role of the high-order methods in dynamic evolution, we give the reference energy $E_{s}$ which is 
calculated by the SAV/CN+SDC scheme in the case of $N_{T}=2048$.
Using the reference value as the baseline, Fig.\,\ref{fig:energy.diff} shows the energy difference of the 
SAV/CN scheme with $N_{T}=64,~128,~256$ and SAV/CN+SDC method with $N_T=32$. 
With the increase of the temporal discretization $N_{T}$, the energy difference decreases. 
However, the energy value obtained by the fourth-order SAV/CN+SDC scheme with
$N_{T}=32$ is closer to the reference value than that of the SAV/CN scheme. 
It is demonstrated that the higher-order method shows more accurate results with
less time discretization points in the dynamic simulation. 
\begin{figure}[htbp]
    \centering
    \includegraphics[scale=0.20]{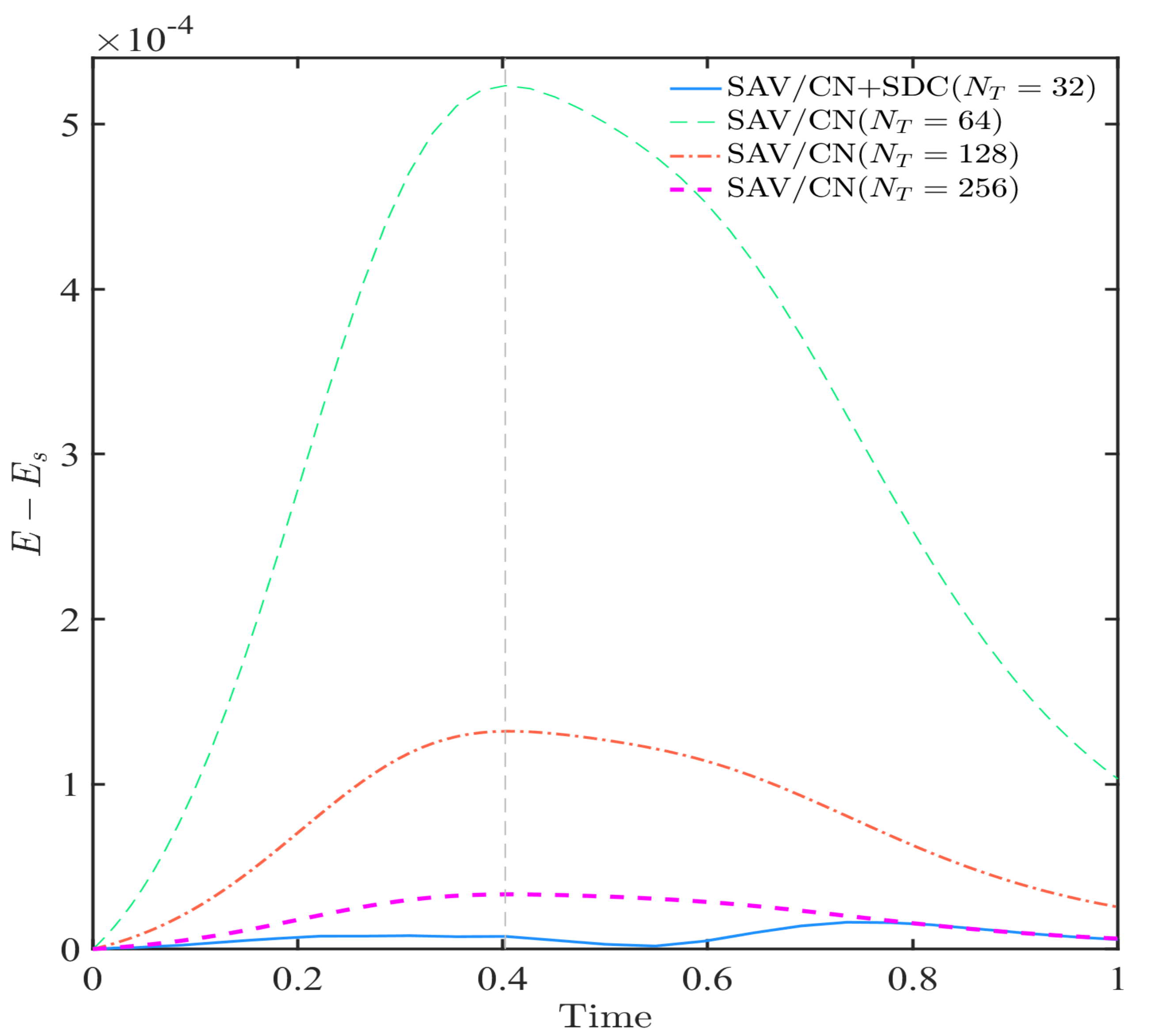}
		\caption{The difference between the numerical energy values and the reference value $E_{s}$. 
		The numerical energy values are computed by the two methods: SAV/CN and SAV/CN+SDC. 
    	The reference energy value is obtained by the SAV/CN+SDC method in the case of $N_{T}=2048$.
		The model parameters are $q_{1} = 1$, $q_{2} = 2\cos(\pi/12)$, $\tep = -2$ and $\tal = 2$.}
	\label{fig:energy.diff}
\end{figure}
We also give the morphologies of the crucial moment $t = 0.4025$ in Fig.\,\ref{fig:energy.diff.moph}. 
The morphology of the reference solution is also set as a reference value to clearly illustrate the differences. 
And the conclusions from these results are consistent with the energy evolution curves. 
\begin{figure}[htbp]
	\centering
	\includegraphics[scale=0.15]{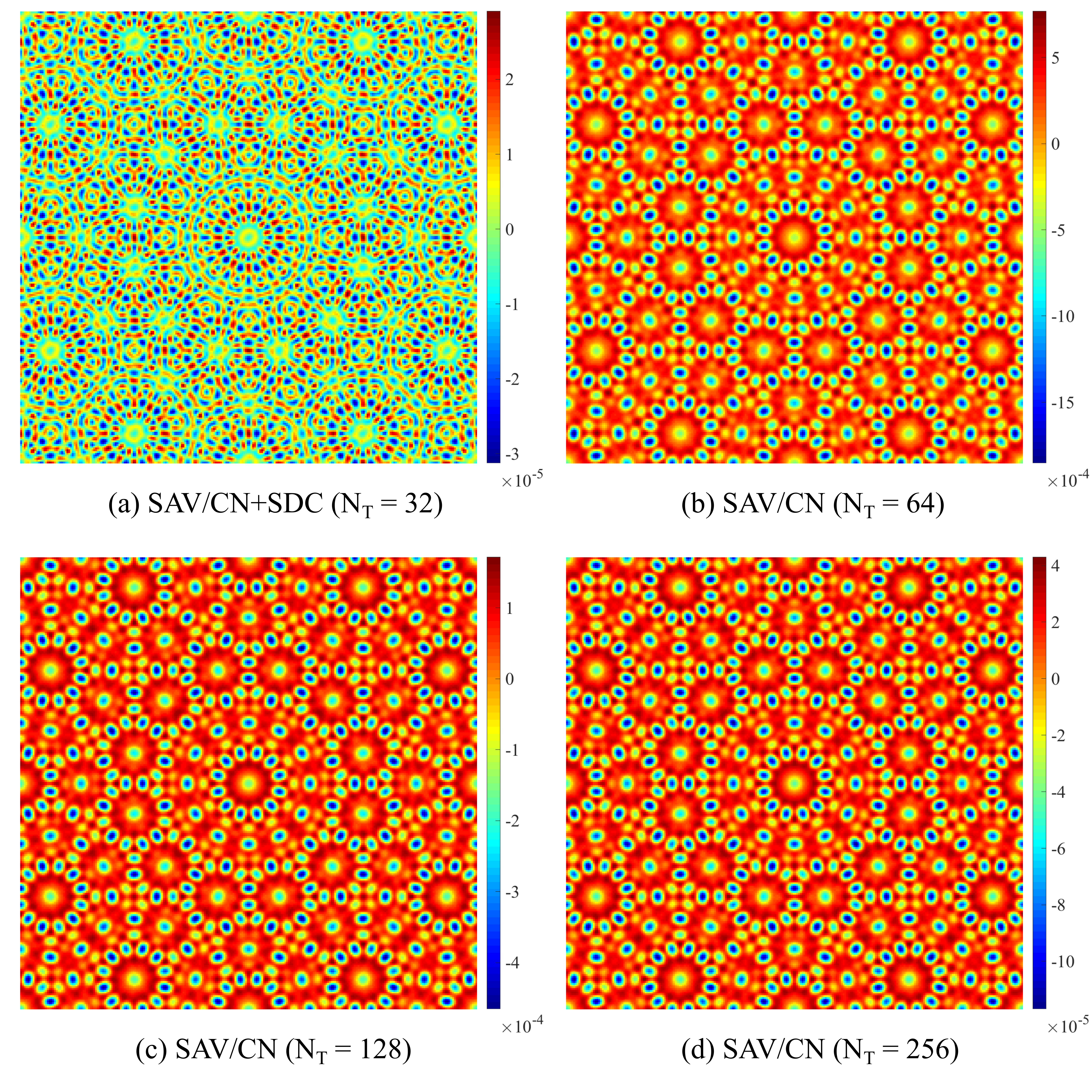}
	\caption{The real morphologies which describe the difference between the numerical solutions
		and the reference value at $t=0.4025$. 
		The reference solution is obtained by the SDC/CN+SDC method in the case of $N_{T}=2048$. 
		The model parameters are set as $q_{1} = 1$, $q_{2} = 2\cos(\pi/12)$, $\tep = -2$ and $\tal = 2$.}
	\label{fig:energy.diff.moph}
\end{figure}

\subsection{The influence of multiple length-scales potential}
\label{subsec:scales}

%% The influence of the multiple length-scales;
In this subsection, we use the dodecagonal quasiperiodic phases as an example to 
investigate the influence of multi-length-scale potentials on the stability of aperiodic structures. 
Concretely, we consider three, four, and five multiple length scale iPFC models. 
The parameters in the potential function are set as $q_{j}=s^{j-1}$, $s=2\cos(\pi/12)$,
$j=1,\cdots,m$, $m=3,4,5$.
The other model parameters are $\tep = -2$ and $\tal = 2$.
%\textcolor{blue}{
In the PM, the $\mathcal{P}$ and $\mathbf{B}$ are consistent with Subsec.\,\ref{subsec:evolution}. 
%The spatial functions are also discretized in $4$-dimensional space with $24^4$ basis functions.
The spatial functions are also discretized in $4$-dimensional space with $24^4$ basis functions.
The SAV/CN approach with $N_{T}=256$ is adopted to simulate the dynamic process. 
Fig.\,\ref{fig:scales} lists the corresponding energy evolution plots, initial values, and stationary states.
The first row presents the energy evolutions of DDQCs with different length-scale potentials. 
The second and third rows give the spectral distributions and real morphologies
of the $m$-length-scale DDQCs, respectively. 
The last row shows the real morphologies of stationary solutions.
From these results, one can see that the three- and four-length-scale potentials
both result in 6-fold symmetric periodic crystal, while the five-length-scale potential
can obtain the 12-fold symmetric quasicrystals. 
Therefore increasing the number of characteristic length-scales in the iPFC model is helpful 
to stabilize the quasicrystals.

%The left and middle patterns both have transition states and the right one directly gets the stationary solution. 
%The left and middle patterns both are periodic structures, but the right one possesses the dodecagonal symmetry. 
%%Combining the Subsec.\,\ref{subsec:evolution}, 
%In particular, except for five length-scales, two, three and four length-scales all are insufficient to stabilize 
%DDQCs under the model parameters. 
%Therefore, increasing the number of length-scales has a great improvement on stabilizing incommensurate multi-length-scale 
%structures. 
\begin{figure}[htbp]
	\centering
	\includegraphics[scale=0.12]{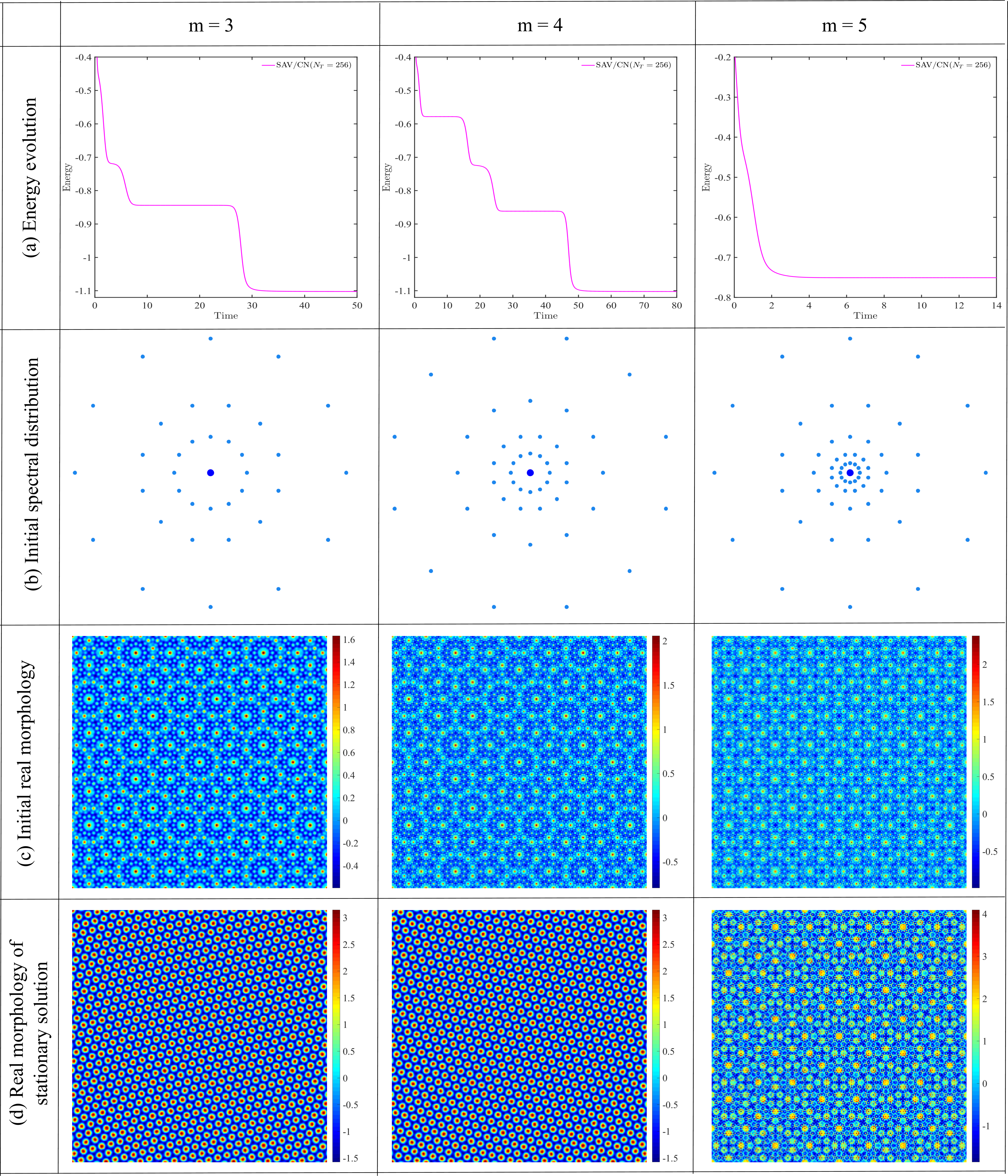}
	\caption{The energy evolutions, initial values and convergence solutions of $m$-length-scale DDQCs under the model 
		parameters $\tep = -2$, $\tal = 2$. 
		The scale parameters are set as $q_{j} = s^{j-1}$, $j=1,\cdots,m$, $s=2\cos(\pi/12)$. }
	\label{fig:scales}
\end{figure}

\section{Summary}
\label{sec:summary}

%% The summary of the above content;
For the iPFC model, we proposed a second-order SAV/CN scheme which is unconditionally energy stable in the almost 
periodic function sense and gave the error estimate. 
Meanwhile, we used the SDC approach to further improve the accuracy of the
second-order scheme to the fourth-order method through a one-step correction.
The PM was applied to discretize spatial functions for computing aperiodic structures to high accuracy.
In numerical simulations, the efficiency of the numerical schemes was demonstrated via the numerical convergence 
rates and the comparison of the dynamic evolutions. 
By comparing the dynamic evolutions of the DDQCs with different length-scales,
we found that increasing the number of characteristic length-scales in the
iPFC model significantly is helpful to stabilize aperiodic structures.


%\begin{appendix}

%\end{appendix}

% You may incorporate your references as follows in your main tex file.
% Using BibTex is not recommended but can be handled.

\end{document}